\documentclass{amsart}
\usepackage[utf8]{inputenc}

\title{Weak$^*$ derived sets of convex sets in duals of non-reflexive spaces}
\author[Z. Silber]{Zdeněk Silber}
\email{zdesil@seznam.cz}
\address{Department of Mathematical Analysis, Faculty of Mathematics and Physics, Charles University, Sokolovská 83, 186 75, Praha 8, Czech Republic}

\usepackage{amsthm}
\usepackage{amsmath}
\usepackage{mathrsfs}
\usepackage{verbatim}
\usepackage{amsfonts}
\usepackage{amssymb}
\usepackage{graphicx}
\usepackage[a-2u]{pdfx}
\usepackage{tikz}
\usepackage[most]{tcolorbox}
\usepackage{esint}
\usepackage{enumitem}
\usetikzlibrary{matrix,arrows}
\graphicspath{{images/}}

\newtheorem*{thm}{Theorem}
\newtheorem{theorem}{Theorem}
\newtheorem{lemma}[theorem]{Lemma}
\newtheorem{prop}[theorem]{Proposition}
\newtheorem*{question}{Question}

\newcommand{\norm}[1]{\left\lVert#1\right\rVert}

\begin{document}

\begin{abstract}
    We investigate weak$^*$ derived sets, that is the sets of weak$^*$ limits of bounded nets, of convex subsets of duals of non-reflexive Banach spaces and their possible iterations.
    We prove that a dual space of any non-reflexive Banach space contains convex subsets of any finite order and a convex subset of order $\omega + 1$.
\end{abstract}

\maketitle

\section{Introduction and formulation of main results}

Let $A$ be a subset of a dual Banach space $X^*$. The weak$^*$ derived set $A^{(1)}$ of the set $A$ is the set of all weak$^*$ limits of bounded convergent nets in $A$, i.e.
\begin{equation*}
    A^{(1)} = \bigcup_{n=1}^\infty \overline{A \cap n B_{X^*}}^{w^*},
\end{equation*}
where $B_{X^*}$ denotes the closet unit ball of $X^*$. If $X$ is separable, bounded sets of $X^*$ are metrizable and the weak$^*$ derived set $A^{(1)}$ coincides with the weak$^*$ sequential closure of $A$.

Recall that a Banach space $X$ is called quasi-reflexive if its canonical embedding into its bidual $X^{**}$ is of finite codimension. All reflexive spaces are also quasi-reflexive and there are non-reflexive quasi-reflexive spaces, e.g. the James' space \cite{james1951}. Reflexivity or quasi-reflexivity of $X$ is closely related to the behaviour of weak$^*$ derived sets of convex subsets of $X^*$:

\begin{thm}
    Let $X$ be a Banach space.
    \begin{enumerate}
        \item $X$ is reflexive if and only if $A^{(1)} = \overline{A}^{w^*}$ for every convex set $A \subseteq X^*$.
        \item $X$ is quasi-reflexive if and only if $A^{(1)} = \overline{A}^{w^*}$ for every subspace $A \subset \subset X^*$.
    \end{enumerate}
\end{thm}

The proof of \textit{(2)}, using the notion of norming subspaces, can be done using the results of \cite{DavisLinden}. The implication from left to right of \textit{(1)} can be easily shown using the Mazur's theorem. The other implication, i.e. the existence of a convex subset $A$ of the dual space of every non-reflexive space for which $A^{(1)} \subsetneq \overline{A}^{w^*}$, is shown in \cite{ostro11}. In \cite{ostro11} Ostrovskii also proved a stronger version of \textit{(2)}: A Banach space $X$ is quasi-reflexive if and only if $A^{(1)} = \overline{A}^{w^*}$ for every absolutely convex set $A \subseteq X^*$.

A convex subset $A$ of a dual Banach space $X^*$ is weak$^*$ closed if and only if it equals its weak$^*$ derived set, i.e. $A = \overline{A}^{w^*}$ if and only if $A = A^{(1)}$. This is a formulation of the Krein-Šmulyan theorem. This inspires the definition of weak$^*$ derived sets of higher orders: For a successor ordinal $\alpha$, the weak$^*$ derived set of $A$ of order $\alpha$ is $A^{(\alpha)} =\left(A^{(\alpha - 1)}\right)^{(1)}$. For a limit ordinal $\alpha$ we define $A^{(\alpha)} = \bigcup_{\beta < \alpha} A^{(\beta)}$. The order of $A$ is the least ordinal $\alpha$, such that $A^{(\alpha)} = A^{(\alpha + 1)}$. We use the convention that $A^{(0)} = A$.

In \cite{ostro91} it is shown that for every non-quasi-reflexive separable Banach space $X$ and every countable ordinal $\alpha$ we can find a subspace $A \subset \subset X^*$ of order $\alpha + 1$. It also holds, that in separable Banach spaces countable non-limit ordinals are the only possible orders of subspaces \cite{HumphreysSimpson}. This gives a complete description of possible orders of subspaces of duals of non-quasi-reflexive separable Banach spaces.

In this paper we prove some partial results regarding orders of convex subsets of duals of non-reflexive Banach spaces. The main results are:

\begin{thm}
    Let $X$ be a non-reflexive Banach space and $n \in \mathbb{N}$. Then there is a convex subset of $X^*$ of order $n$.
\end{thm}

\begin{thm}
    Let $X$ be a non-reflexive Banach space. Then there is a convex subset of $X^*$ of order $\omega + 1$.
\end{thm}

These results are proved below in Theorems \ref{ThmRadn} and \ref{ThmRadOmega}. Note that we can restrict ourselves to the case non-reflexive quasi-reflexive Banach spaces. In the case of reflexive spaces the only possible orders of convex sets are $0$, if the set is already weak$^*$ closed, or $1$, if the set is not weak$^*$ closed. The case of non-quasi-reflexive separable spaces is already solved in \cite{ostro91}. Both of these results use a modified construction of Ostrovskii used in \cite{ostro11}.

\section{Proofs of main results}

\begin{lemma} \label{LemmaPrevod}
    Let $X$ be a Banach space and $Z \subset \subset X$ its closed subspace. Denote by $E:Z \rightarrow X$ the identity embedding. Then for every ordinal $\alpha$ and $A \subseteq Z^*$ we have
    \begin{equation*}
        (E^*)^{-1}(A^{(\alpha)}) = ((E^*)^{-1}(A))^{(\alpha)}.
    \end{equation*}
\end{lemma}

This lemma is proved in \cite[Lemma 1]{ostro11} for $\alpha = 1$. For general $\alpha$ the lemma follows by transfinite induction. Note that the weak$^*$ derived set $A^{(\alpha)}$ is taken in $Z^*$ and $((E^*)^{-1}(A))^{(\alpha)}$ is taken in $X^*$.

\begin{lemma} \label{LemmaBaze}
    Let $X$ be a non-reflexive Banach space. Then $X$ contains a seminormalized basic sequence $(z_n)_{n=1}^\infty$ which has bounded partial sums.
\end{lemma}

\begin{proof}
    This lemma is proved for a non-reflexive space with a basis in \cite[Theorem 3, ($1^\circ \Leftrightarrow 3^\circ$)]{singer1962}. As any non-reflexive space contains a non-reflexive subspace with a basis \cite[Theorem 1]{pelczynski1962}, the lemma follows.
\end{proof}

For the rest of this paper we will work with the Banach space $Z$ with seminormalized basis $(z_n)_{n=1}^\infty$ with bounded partial sums, i.e. there are constants $C,C_1,C_2>0$ such that $||\sum_{n=1}^N z_n || \leq C$ for all $N \in \mathbb{N}$ and $C_1 \leq ||z_n|| \leq C_2$ for all $n \in \mathbb{N}$. Let us denote by $(z_n^*)_{n=1}^\infty$ the biorthogonal functionals of $(z_n)_{n=1}^\infty$ and by $K$ the positive cone of $Z^*$. That is the weak$^*$ closed convex set
\begin{equation*}
    K = \{z^* \in Z^*; \; z^*(z_j) \geq 0 \text{ for each } j \in \mathbb{N}\}.
\end{equation*}

Note that as $Z$ is separable, weak$^*$ derived sets in $Z^*$ coincide with weak$^*$ sequential closures. Also note, that as the basis $(z_n)_{n=1}^\infty$ is seminormalized, we get that $z_n^* \overset{w^*}{\longrightarrow}0$.

\begin{lemma} \label{LemmaKuzel}
    For every $z^* \in K$ we have $z^* = \sum_{n=1}^\infty z^*(z_n) z_n^*$, where the series converges absolutely. Further, we have $||z^*|| \geq C^{-1} \sum_{n=1}^\infty z^*(z_n)$.
\end{lemma}

\begin{proof}
    For each $N \in \mathbb{N}$ we have
    \begin{equation*}
        ||z^*|| \geq C^{-1} z^* \left(\sum_{n=1}^N z_n \right) = C^{-1} \sum_{n=1}^N z^*(z_n).
    \end{equation*}
    Hence, as $z^*(z_n) \geq 0$ for each $n \in \mathbb{N}$, we get $||z^*|| \geq C^{-1} \sum_{n=1}^\infty z^*(z_n)$ and the series $\sum_{n=1}^\infty z^*(z_n) z_n^*$ converges absolutely in $Z^*$. As it also converges to $z^*$ in the weak$^*$ topology, we get that $\sum_{n=1}^\infty z^*(z_n) z_n^* = z^*$.
\end{proof}

Now, let us partition $\mathbb{N}$ into countably many subsequences: There will be the set $\mathbf{N}_0 = \{i_1 < i_2 < \dots \}$. Then for each $n \in \mathbb{N}$ there will be the set $\mathbf{N}(i_n)$, for each $j_1 \in \mathbf{N}(i_n)$ there will be the set $\mathbf{N}(i_n,j_1)$ and so on up to for each $j_n \in \mathbf{N}(i_n,j_1,\dots,j_{n-1})$ there will be the set $\mathbf{N}(i_n,j_1,\dots,j_n)$.

Fix a sequence of positive numbers $(\beta_k)_{k=1}^\infty$, such that $0\neq \beta_k \nearrow \infty$, and a sequence $(\alpha_k)_{k=1}^\infty$ of numbers in the interval $[0,1)$, such that for each $n \in \mathbb{N}$ we have that $(\alpha_{j_1})_{j_1 \in \mathbf{N}(i_n)}$ is a sequence increasing monotonically to $1$ with the first element equal to $0$. We will say that a finite sequence of positive integers $(n_1,\dots,n_k)$ is admissible if $n_1 \in \mathbf{N}_0$ and for each $2 \leq i \leq k$ we have that $n_i \in \mathbf{N}(n_1,\dots,n_{i-1})$.

For each $n \in \mathbb{N}$ define

\begin{align*}
    A_n &= \operatorname{conv} \left\{ \alpha_{j_1} z_{i_n}^* + \sum_{k=1}^n \beta_{j_k} z^*_{j_{k+1}}; \; (i_n,j_1,\dots,j_{n+1}) \text{ is admissible} \right\}
\end{align*}
and, moreover, define

\begin{align*}
    A &= \operatorname{conv} \bigcup_{n=1}^\infty A_n.
\end{align*}

Let us further denote by $\mathbf{N}_n$ the support of $A_n$, i.e.
\begin{align*}
    \mathbf{N}_n = \{i_n\} \cup \bigcup \{ \mathbf{N}(i_n,j_1)  \cup \mathbf{N}(i_n,j_1,j_2) \cup & \cdots \cup \mathbf{N}(i_n,j_1,\dots,j_{n}); \\
    & (i_n,j_1,\dots,j_n) \text{ is admissible} \}.
\end{align*}

Later we will prove that those $A_n$'s are the desired convex sets of order $n+1$ and $A$ is the desired convex set of order $\omega + 1$.

\begin{prop} \label{PropCharakterizaceAn}
    $A_n$ is the set of those $x^* \in Z^*$ which have finite support in $\mathbf{N}_n$ and which satisfy the following equations:
    \begin{align*}
        1 &= \sum_{\substack{j_1 \in \mathbf{N}(i_n) \\ j_2 \in \mathbf{N}(i_n,j_1)}} \frac{x^*(z_{j_2})}{\beta_{j_1}} \hfill & \\
        x^*(z_{i_n}) &= \sum_{\substack{j_1 \in \mathbf{N}(i_n) \\ j_2 \in \mathbf{N}(i_n,j_1)}} \frac{x^*(z_{j_2}) \alpha_{j_1}}{\beta_{j_1}} \hfill & \\
        x^*(z_{j_2}) &= \sum_{\substack{j_3 \in \mathbf{N}(i_n,j_1,j_2)}} \frac{x^*(z_{j_3}) \beta_{j_1}}{\beta_{j_2}} \hfill & j_1 \in \mathbf{N}(i_n), \; j_2 \in \mathbf{N}(i_n,j_1) \\
        x^*(z_{j_3}) &= \sum_{\substack{j_4 \in \mathbf{N}(i_n,j_1,j_2,j_3)}} \frac{x^*(z_{j_4}) \beta_{j_2}}{\beta_{j_3}} \hfill & j_1 \in \mathbf{N}(i_n), \dots, \; j_3 \in \mathbf{N}(i_n,j_1,j_2) \\
        & \vdots \\
        x^*(z_{j_n}) &= \sum_{\substack{j_{n+1} \in \mathbf{N}(i_n,j_1,\dots,j_n)}} \frac{x^*(z_{j_{n+1}}) \beta_{j_{n-1}}}{\beta_{j_n}} \hfill & j_1 \in \mathbf{N}(i_n),  \dots, \; j_n \in \mathbf{N}(i_n,j_1,\dots,j_{n-1}).
    \end{align*}
\end{prop}

\begin{proof}
    Each element of $A_n$ has finite support in $\mathbf{N}_n$ and satisfies the required equations, as the vectors $\alpha_{j_1} z^*_{i_n} + \sum_{k=1}^n \beta_{j_k} z^*_{j_{k+1}}$ satisfy them and the validity of these equations is preserved by taking convex combinations. To prove the converse inclusion, let us have $x^* \in Z^*$ with finite support in $\mathbf{N}_n$ and satisfying these equations. Set $c_{j_k} = x^*(z_{j_k})^{-1}$, if $x^*(z_{j_k}) \neq 0$, and $c_{j_k} = 0$ otherwise. Then for each admissible $(i_n,j_1,\dots,j_n)$ we have
    \begin{align*}
        x_{j_n}^* &:= \sum_{\substack{j_{n+1} \in \mathbf{N}(i_n,j_1,\dots,j_n)}} \frac{c_{j_n} x^*(z_{j_{n+1}}) \beta_{j_{n-1}}}{\beta_{j_n}}  \left( \alpha_{j_1} z_{i_n}^* + \sum_{k=1}^n \beta_{j_k} z^*_{j_{k+1}} \right) \in A_n \\
        x_{j_{n-1}}^* &:= \sum_{\substack{j_{n} \in \mathbf{N}(i_n,j_1,\dots,j_{n-1})}} \frac{c_{j_{n-1}} x^*(z_{j_{n}}) \beta_{j_{n-2}}}{\beta_{j_{n-1}}}  x_{j_n}^* \in A_n \\
        &\vdots \\
        x_{j_2}^* &:= \sum_{\substack{j_3 \in \mathbf{N}(i_n,j_1,j_2)}} \frac{c_{j_2} x^*(z_{j_3}) \beta_{j_1}}{\beta_{j_2}}  x_{j_3}^* \in A_n \\
        y^* &:= \sum_{\substack{j_1 \in \mathbf{N}(i_n) \\ j_2 \in \mathbf{N}(i_n,j_1)}} \frac{x^*(z_{j_2})}{\beta_{j_1}} x_{j_2}^* \in A_n.
    \end{align*}
    The fact that they are elements $A_n$ follows from convexity of $A_n$ and the choice of $c_{j_k}$. Hence, we just need to show that $y^* = x^*$. If $m \notin \mathbf{N}_n$ we have that $x^*(z_m) = y^*(z_m) = 0$. Let $m \leq n$ and fix an admissible $(i_n,j_1,\dots,j_m)$. Then for each admissible $(i_n,\widetilde{j_1},\dots,\widetilde{j_n})$ we have
    \begin{align*}
        x^*_{\widetilde{j_n}}(z_{j_m}) = \sum_{\substack{\widetilde{j_{n+1}} \in \mathbf{N}(i_n,\widetilde{j_1},\dots,\widetilde{j_n})}} \frac{c_{\widetilde{j_n}} x^*(z_{\widetilde{j_{n+1}}}) \beta_{\widetilde{j_{n-1}}}}{\beta_{\widetilde{j_n}}} \beta_{j_{m-1}}
    \end{align*}
    if $\widetilde{j_m} = j_m$ (and therefore $\widetilde{j_i} = j_i$ for each $i \leq m$) and $x^*_{\widetilde{j_n}}(z_{j_m}) = 0$ otherwise. Hence, for each admissible $(i_n,\widetilde{j_1},\dots,\widetilde{j_{n-1}})$, we have
    \begin{align*}
        x^*_{\widetilde{j_{n-1}}}(z_{j_m}) = \sum_{\substack{\widetilde{j_n} \in \mathbf{N}(i_n,\widetilde{j_1},\dots,\widetilde{j_{n-1}}) \\ \widetilde{j_{n+1}} \in \mathbf{N}(i_n,\widetilde{j_1},\dots,\widetilde{j_n})}} \frac{c_{\widetilde{j_{n-1}}} x^*(z_{\widetilde{j_{n}}}) \beta_{\widetilde{j_{n-2}}}}{\beta_{\widetilde{j_{n-1}}}} \frac{c_{\widetilde{j_n}} x^*(z_{\widetilde{j_{n+1}}}) \beta_{\widetilde{j_{n-1}}}}{\beta_{\widetilde{j_n}}} \beta_{j_{m-1}}
    \end{align*}
    if $\widetilde{j_m} = j_m$ and $x^*_{\widetilde{j_{n-1}}}(z_{j_m}) = 0$ otherwise. Iterating this, we get for $m + 1 \leq k \leq n$ and admissible $(i_n,\widetilde{j_1},\dots,\widetilde{j_k})$
    \begin{align*}
        x^*_{\widetilde{j_{k}}}(z_{j_m}) = \sum_{\substack{\widetilde{j_{k+1}} \in \mathbf{N}(i_n,\widetilde{j_1},\dots,\widetilde{j_{k}}), \dots, \\ \widetilde{j_{n+1}} \in \mathbf{N}(i_n,\widetilde{j_1},\dots,\widetilde{j_{n}})}} \frac{c_{\widetilde{j_{k}}} x^*(z_{\widetilde{j_{k+1}}}) \beta_{\widetilde{j_{k-1}}}}{\beta_{\widetilde{j_{k}}}} \cdots \frac{c_{\widetilde{j_n}} x^*(z_{\widetilde{j_{n+1}}}) \beta_{\widetilde{j_{n-1}}}}{\beta_{\widetilde{j_n}}} \beta_{j_{m-1}}
    \end{align*}
    if $\widetilde{j_k} = j_k$ and $x^*_{\widetilde{j_{k}}}(z_{j_m}) = 0$ otherwise. Hence, we get
    \begin{align*}
        x_{j_m}^* (z_{j_m}) = \sum_{\substack{j_{m+1} \in \mathbf{N}(i_n,j_1,\dots,j_{m}), \dots, \\ j_{n+1} \in \mathbf{N}(i_n,j_1,\dots,j_{n})}} \frac{c_{j_m} x^*(z_{j_{m+1}}) \beta_{j_{m-1}}}{\beta_{j_m}} \cdots \frac{c_{j_n} x^*(z_{j_{n+1}})\beta_{j_{n-1}}}{\beta_{j_n}} \beta_{j_{m-1}}
    \end{align*}
    and for admissible $(i_n,\widetilde{j_1},\dots,\widetilde{j_m})$ such that $j_m \neq \widetilde{j_m}$ we get $x_{\widetilde{j_m}}^*(z_{j_m}) = 0$. We can then inductively prove that if $2 \leq k \leq m - 1$, the only admissible $(i_n,\widetilde{j_1},\dots,\widetilde{j_k})$ with nonzero $x^*_{\widetilde{j_k}}(z_{j_m})$ are the initial segments of $(i_n,j_1,\dots,j_m)$ and for them we have
    \begin{align*}
        x^*_{j_k}(z_{j_m}) = \sum_{\substack{j_{m+1} \in \mathbf{N}(i_n,j_1,\dots,j_{m}), \dots, \\ j_{n+1} \in \mathbf{N}(i_n,j_1,\dots,j_{n})}} \frac{c_{j_k} x^*(z_{j_{k+1}}) \beta_{j_{k-1}}}{\beta_{j_k}} \cdots \frac{c_{j_n} x^*(z_{j_{n+1}})\beta_{j_{n-1}}}{\beta_{j_n}} \beta_{j_{m-1}}.
    \end{align*}
    Then, as $c_{j_k} = x^*(z_{j_k})^{-1}$, we can finally show that
    \begin{align*}
        y^*(z_{j_m}) &= \sum_{\substack{j_{m+1} \in \mathbf{N}(i_n,j_1,\dots,j_{m}), \dots, \\ j_{n+1} \in \mathbf{N}(i_n,j_1,\dots,j_{n})}} \frac{x^*(z_{j_2})}{\beta_{j_1}} \frac{c_{j_2} x^*(z_{j_3}) \beta_{j_1}}{\beta_{j_2}} \cdots \frac{c_{j_n} x^*(z_{j_{n+1}})\beta_{j_{n-1}}}{\beta_{j_n}} \beta_{j_{m-1}} \\
        &= \sum_{\substack{j_{m+1} \in \mathbf{N}(i_n,j_1,\dots,j_{m}), \dots, \\ j_{n+1} \in \mathbf{N}(i_n,j_1,\dots,j_{n})}} \frac{\beta_{j_{m-1}}}{\beta_{j_n}} x^*(z_{j_{n+1}}).
    \end{align*}
    Now, by consecutive application of the equations of the proposition, we get 
    \begin{align*}
        x^*(z_{j_m}) &= \sum_{\substack{j_{m+1} \in \mathbf{N}(i_n,j_1,\dots,j_{m})}} \frac{x^*(z_{j_{m+1}}) \beta_{j_{m-1}}}{\beta_{j_m}} \\
        &= \sum_{\substack{j_{m+1} \in \mathbf{N}(i_n,j_1,\dots,j_{m}) \\ j_{m+2} \in \mathbf{N}(i_n,j_1,\dots,j_{m+1})}} \frac{x^*(z_{j_{m+2}}) \beta_{j_{m-1}}}{\beta_{j_{m}}} \frac{\beta_{j_{m}}}{\beta_{j_{m+1}}} \\ &=
        \sum_{\substack{j_{m+1} \in \mathbf{N}(i_n,j_1,\dots,j_{m}) \\ j_{m+2} \in \mathbf{N}(i_n,j_1,\dots,j_{m+1})}} \frac{x^*(z_{j_{m+2}}) \beta_{j_{m-1}}}{\beta_{j_{m+1}}} = \cdots \\
        & \cdots = \sum_{\substack{j_{m+1} \in \mathbf{N}(i_n,j_1,\dots,j_{m}), \dots, \\ j_{n+1} \in \mathbf{N}(i_n,j_1,\dots,j_{n})}} \frac{\beta_{j_{m-1}}}{\beta_{j_n}} x^*(z_{j_{n+1}}).
    \end{align*}
    Hence, $x^*(z_{j_m}) = y^*(z_{j_m})$. Analogically
    \begin{align*}
        x^*(z_{i_n}) = y^*(z_{i_n}) =\sum_{\substack{j_{m+1} \in \mathbf{N}(i_n,j_1,\dots,j_{m}), \dots, \\ j_{n+1} \in \mathbf{N}(i_n,j_1,\dots,j_{n})}} \frac{\alpha_{j_1}}{\beta_{j_n}} x^*(z_{j_{n+1}})
    \end{align*}
    and for admissible $(i_n,j_1,\dots,j_{n+1})$ we have that $x^*(z_{j_{n+1}}) = y^*(z_{j_{n+1}})$. Hence, $x^* = y^*$ and we are done.
\end{proof}

\begin{prop} \label{PropRovnostiAnm}
    Let $0 \leq m \leq n-1$ and $x^*$ be an element of $A_n^{(m)}$. Then $x^*$ satisfies the equations of Proposition \ref{PropCharakterizaceAn} possibly except for the equations on the bottom $m$ lines. Precisely:
    \begin{align*}
        1 &= \sum_{\substack{j_1 \in \mathbf{N}(i_n) \\ j_2 \in \mathbf{N}(i_n,j_1)}} \frac{x^*(z_{j_2})}{\beta_{j_1}} \\
        x^*(z_{i_n}) &= \sum_{\substack{j_1 \in \mathbf{N}(i_n) \\ j_2 \in \mathbf{N}(i_n,j_1)}} \frac{x^*(z_{j_2}) \alpha_{j_1}}{\beta_{j_1}}
    \end{align*}
    and for $2 \leq k \leq n-m$ and admissible $(i_n,j_1,\dots,j_k)$
    \begin{align*}
        x^*(z_{j_k}) &= \sum_{\substack{j_{k+1} \in \mathbf{N}(i_n,j_1,\dots,j_k)}} \frac{x^*(z_{j_{k+1}}) \beta_{j_{k-1}}}{\beta_{j_k}}.
    \end{align*}
\end{prop}

\begin{proof}
    We shall proceed by induction. We have already shown in Proposition \ref{PropCharakterizaceAn} that the proposition holds for $m=0$. Now, let us suppose that the proposition holds for $m-1$ and take $x^* \in A_n^{(m)}$. There is a sequence $(x_i^*)_{i=1}^\infty$ in $A_n^{(m-1)}$, such that $x_i^* \overset{w^*}{\rightarrow} x^*$. Take admissible $(i_n,j_1,\dots,j_{k})$ where $k \leq n-m$. Suppose, for a contradiction, that
    \begin{align*}
        x^*(z_{j_k}) \neq \sum_{\substack{j_{k+1} \in \mathbf{N}(i_n,j_1,\dots,j_k)}} \frac{x^*(z_{j_{k+1}}) \beta_{j_{k-1}}}{\beta_{j_k}}.
    \end{align*}
    Take
    \begin{align*}
        \delta = x^*(z_{j_k}) - \sum_{\substack{j_{k+1} \in \mathbf{N}(i_n,j_1,\dots,j_k)}} \frac{x^*(z_{j_{k+1}}) \beta_{j_{k-1}}}{\beta_{j_k}}.
    \end{align*}
    By the induction hypothesis, as $x_i^* \in A^{(m-1)}_n$, we have
    \begin{align*}
        x_i^*(z_{j_{k+1}}) &= \sum_{\substack{j_{k+2} \in \mathbf{N}(i_n,j_1,\dots,j_{k+1})}} \frac{x_i^*(z_{j_{k+2}}) \beta_{j_{k}}}{\beta_{j_{k+1}}} \\
        x_i^*(z_{j_k}) &= \sum_{\substack{j_{k+1} \in \mathbf{N}(i_n,j_1,\dots,j_k)}} \frac{x_i^*(z_{j_{k+1}}) \beta_{j_{k-1}}}{\beta_{j_k}}.
    \end{align*}
    Hence, by Fatou's lemma, we get that $\delta \geq 0$ and as $\delta$ is nonzero we get $\delta > 0$.
    
    For $c>0$ take
    \begin{align*}
         F_c &= \{j_{k+1} \in \mathbf{N}(i_n,j_1,\dots,j_k); \; \beta_{j_{k+1}} \leq c\} \\ G_c &= \mathbf{N}(i_n,j_1,\dots,j_k) \setminus F_c.
    \end{align*}
    Then $F_c$ is a finite set and $x_i^* \overset{w^*}{\rightarrow} x^*$, therefore there is $i_0 \in \mathbb{N}$, such that for $i \geq i_0$ we have
    \begin{align*}
        \sum_{j_{k+1} \in F_c} \frac{x_i^*(z_{j_{k+1}}) \beta_{j_{k-1}}}{\beta_{j_k}} &< \sum_{j_{k+1} \in F_c} \frac{x^*(z_{j_{k+1}}) \beta_{j_{k-1}}}{\beta_{j_k}} + \delta/2 \\
        &\leq \sum_{j_{k+1} \in \mathbf{N}(i_n,j_1,\dots,j_k)} \frac{x^*(z_{j_{k+1}}) \beta_{j_{k-1}}}{\beta_{j_k}} + \delta/2 \\
        &= x^*(z_{j_k}) - \delta/2,
    \end{align*}
    and therefore
    \begin{align*}
        \sum_{\substack{j_{k+1} \in F_c \\ j_{k+2} \in \mathbf{N}(i_n,j_1,\dots, j_{k+1})}}  \frac{x_i^*(z_{j_{k+2}}) \beta_{j_{k-1}}}{\beta_{j_{k+1}}} = \sum_{j_{k+1} \in F_c} \frac{x_i^*(z_{j_{k+1}}) \beta_{j_{k-1}}}{\beta_{j_k}} < x^*(z_{j_k}) - \delta/2.
    \end{align*}
    Then there is $i_1 \geq i_0$ such  that
    \begin{align*}
        \sum_{\substack{j_{k+1} \in G_c \\ j_{k+2} \in \mathbf{N}(i_n,j_1,\dots, j_{k+1})}}  \frac{x_{i_1}^*(z_{j_{k+2}}) \beta_{j_{k-1}}}{\beta_{j_{k+1}}} = \sum_{j_{k+1} \in G_c} \frac{x_{i_1}^*(z_{j_{k+1}}) \beta_{j_{k-1}}}{\beta_{j_k}} > \delta/4,
    \end{align*}
    as otherwise $x_i^*(z_{j_k}) < x^*(z_{j_k}) - \delta/4$ for all $i \geq i_0$, which would contradict $x_i^* \overset{*}{\rightarrow} x^*$. But then it follows from Lemma \ref{LemmaKuzel} that
    \begin{align*}
         \norm{x_{i_1}^*} \geq C^{-1} \sum_{\substack{j_{k+1} \in G_c \\ j_{k+2} \in \mathbf{N}(i_n,j_1,\dots, j_{k+1})}} x_{i_1}^*(z_{j_{k+2}}) > C^{-1} \beta_{j_{k-1}}^{-1} c \; \delta/4.
    \end{align*}
    As $c>0$ was chosen arbitrarily, we get that $(x_i^*)_{i=1}^\infty$ is unbounded. But this contradicts the Banach-Steinhaus theorem. Hence,
    \begin{align*}
        x^*(z_{j_k}) = \sum_{\substack{j_{k+1} \in \mathbf{N}(i_n,j_1,\dots,j_k)}} \frac{x^*(z_{j_{k+1}}) \beta_{j_{k-1}}}{\beta_{j_k}}.
    \end{align*}
    
    Now suppose for a contradiction that
    \begin{align*}
        x^*(z_{i_n}) - \sum_{\substack{j_1 \in \mathbf{N}(i_n) \\ j_2 \in \mathbf{N}(i_n,j_1)}} \frac{x^*(z_{j_2}) \alpha_{j_1}}{\beta_{j_1}} = \delta \neq 0.
    \end{align*}
    By the same argument as above we get that $\delta > 0$. As $x_i^* \in A^{(m-1)}_n$, we get by the induction hypothesis that
    \begin{align*}
        x_i^*(z_{i_n}) &= \sum_{\substack{j_1 \in \mathbf{N}(i_n) \\ j_2 \in \mathbf{N}(i_n,j_1)}} \frac{x_i^*(z_{j_2}) \alpha_{j_1}}{\beta_{j_1}}
    \end{align*}
    and for admissible $(i_n,j_1,j_2)$
    \begin{align*}
        x_i^*(z_{j_2}) &= \sum_{\substack{j_{3} \in \mathbf{N}(i_n,j_1,j_2)}} \frac{x^*(z_{j_{3}}) \beta_{j_{1}}}{\beta_{j_2}}.
    \end{align*}
    Now, for $c>0$ set
    \begin{align*} 
        F_c &= \{ j_2; \; \beta_{j_2} \leq c \text{ and } (i_n,j_1,j_2) \text{ is admissible} \} \\
        G_c &= \bigcup \{\mathbf{N}(i_n,j_1); \; j_1 \in \mathbf{N}(i_n)\} \setminus F_c.
    \end{align*}
    Then, as $F_c$ is finite and $x_i^* \overset{w^*}{\rightarrow} x^*$, we get in the same way as above that there is $i_0 \in \mathbb{N}$ such that for $i \geq i_0$
    \begin{align*}
        \sum_{\substack{j_2 \in F_c}} \frac{x_i^*(z_{j_2}) \alpha_{j_1}}{\beta_{j_1}} < x^*(z_{i_n}) - \delta/2,
    \end{align*}
    and therefore there is $i_1 \geq i_0$ such that
    \begin{align*}
       \sum_{\substack{j_{2} \in G_c \\ j_3 \in \mathbf{N}(i_n,j_1,j_2)}} \frac{x_{i_1}^*(z_{j_{3}}) \alpha_{j_{1}}}{\beta_{j_2}} = \sum_{\substack{j_2 \in G_c}} \frac{x_{i_1}^*(z_{j_2}) \alpha_{j_1}}{\beta_{j_1}} > \delta/4.
    \end{align*}
    But then again by Lemma \ref{LemmaKuzel} we have, for $j$ being the second element of $\mathbf{N}(i_n)$,
    \begin{align*}
        ||x_{i_1}^*|| \geq C^{-1} \alpha_j^{-1} c \; \delta/4,
    \end{align*}
    which contradicts boundedness of the sequence $(x_i^*)_{i=1}^\infty$. Hence,
    \begin{align*}
        x^*(z_{i_n}) = \sum_{\substack{j_1 \in \mathbf{N}(i_n) \\ j_2 \in \mathbf{N}(i_n,j_1)}} \frac{x^*(z_{j_2}) \alpha_{j_1}}{\beta_{j_1}}.
    \end{align*}
    In exactly the same way we can show that
    \begin{align*}
        1 - \sum_{\substack{j_1 \in \mathbf{N}(i_n) \\ j_2 \in \mathbf{N}(i_n,j_1)}} \frac{x^*(z_{j_2})}{\beta_{j_1}} = \delta > 0
    \end{align*}
    leads to the fact that for all $c>0$ there is $i_1 \in \mathbb{N}$ such that
    \begin{align*}
        ||x_{i_1}^*|| \geq C^{-1} c \; \delta/4,
    \end{align*}
    and contradicts boundedness of the sequence $(x_i^*)_{i=1}^\infty$. Hence,
    \begin{align*}
        1 = \sum_{\substack{j_1 \in \mathbf{N}(i_n) \\ j_2 \in \mathbf{N}(i_n,j_1)}} \frac{x^*(z_{j_2})}{\beta_{j_1}}.
    \end{align*}
\end{proof}

\begin{lemma} \label{LemmaOrderAnZdola}
    The order of $A_n$ is at least $n+1$. Specifically $z_{i_n}^* \in A_n^{(n+1)} \setminus A_n^{(n)}$.
\end{lemma}

\begin{proof}
    First, observe that $z^*_{i_n} \in A_n^{(n+1)}$ as
    \begin{align*}
        z_{i_n}^* = w^* \lim_{j_{1}} \cdots w^* \lim_{j_{n+1}} \; \left( \alpha_{j_1} z_{i_n}^* + \sum_{k=1}^n \beta_{j_k} z^*_{j_{k+1}} \right).
    \end{align*}
    Now, suppose for a contradiction that $z_{i_n}^* \in A_n^{(n)}$. There is a sequence $(x_i^*)_{i=1}^\infty$ in $A_n^{(n-1)}$ which weak$^*$ converges to $z_{i_n}^*$. By Proposition \ref{PropRovnostiAnm} we have
    \begin{align*}
        x_i^*(z_{i_n}) &= \sum_{\substack{j_1 \in \mathbf{N}(i_n) \\ j_2 \in \mathbf{N}(i_n,j_1)}} \frac{x_i^*(z_{j_2}) \alpha_{j_1}}{\beta_{j_1}} \\
        1 &= \sum_{\substack{j_1 \in \mathbf{N}(i_n) \\ j_2 \in \mathbf{N}(i_n,j_1)}} \frac{x_i^*(z_{j_2})}{\beta_{j_1}},
    \end{align*}
    Now fix an arbitrary $M \in \mathbb{N}$, then
    \begin{align*}
        1 = z_{i_n}^*(z_{i_n}) &= \lim_i x_i^*(z_{i_n}) = \lim_i \left( \sum_{\substack{j_1 \in \mathbf{N}(i_n), \; j_1 \leq M \\ j_2 \in \mathbf{N}(i_n,j_1)}} \frac{x_i^*(z_{j_2}) \alpha_{j_1}}{\beta_{j_1}} + \sum_{\substack{j_1 \in \mathbf{N}(i_n), \; j_1 > M \\ j_2 \in \mathbf{N}(i_n,j_1)}} \frac{x_i^*(z_{j_2}) \alpha_{j_1}}{\beta_{j_1}} \right)  \\
        &\leq \liminf_i \left( \alpha_M \sum_{\substack{j_1 \in \mathbf{N}(i_n), \; j_1 \leq M \\ j_2 \in \mathbf{N}(i_n,j_1)}} \frac{x_i^*(z_{j_2})}{\beta_{j_1}} + \sum_{\substack{j_1 \in \mathbf{N}(i_n), \; j_1 > M \\ j_2 \in \mathbf{N}(i_n,j_1)}} \frac{x_i^*(z_{j_2})}{\beta_{j_1}} \right) \\
        &= \liminf_i \left( \sum_{\substack{j_1 \in \mathbf{N}(i_n) \\ j_2 \in \mathbf{N}(i_n,j_1)}} \frac{x_i^*(z_{j_2})}{\beta_{j_1}} + (\alpha_M - 1) \sum_{\substack{j_1 \in \mathbf{N}(i_n), \; j_1 \leq M \\ j_2 \in \mathbf{N}(i_n,j_1)}} \frac{x_i^*(z_{j_2})}{\beta_{j_1}} \right) \\
        &\leq 1 + (\alpha_M - 1) \; \liminf_i \left( \sum_{\substack{j_1 \in \mathbf{N}(i_n), \; j_1 \leq M \\ j_2 \in \mathbf{N}(i_n,j_1)}} \frac{x_i^*(z_{j_2})}{\beta_{j_1}} \right).
    \end{align*}
    As $\alpha_M - 1 < 0$, we get, up to passing to a subsequence if necessary, 
    \begin{align*}
        \lim_i \left( \sum_{\substack{j_1 \in \mathbf{N}(i_n), \; j_1 \leq M \\ j_2 \in \mathbf{N}(i_n,j_1)}} \frac{x_i^*(z_{j_2})}{\beta_{j_1}} \right) = 0.
    \end{align*}
    Then
    \begin{align*}
        \lim_i & \left( \sum_{\substack{j_1 \in \mathbf{N}(i_n), \; j_1 > M \\ j_2 \in \mathbf{N}(i_n,j_1)}} \frac{x_i^*(z_{j_2})}{\beta_{j_1}} \right) = \lim_i \left( \sum_{\substack{j_1 \in \mathbf{N}(i_n), \\ j_2 \in \mathbf{N}(i_n,j_1)}} \frac{x_i^*(z_{j_2})}{\beta_{j_1}} \right) = 1.
    \end{align*}
    Hence, there is $i \in \mathbb{N}$, such that
    \begin{align*}
        \sum_{\substack{j_1 \in \mathbf{N}(i_n), \; j_1 > M \\ j_2 \in \mathbf{N}(i_n,j_1)}} \frac{x_i^*(z_{j_2})}{\beta_{j_1}} > 1/2.
    \end{align*}
    But then it follows from Lemma \ref{LemmaKuzel} that
    \begin{align*}
        \norm{x_i}^* \geq C^{-1} \sum_{\substack{j_1 \in \mathbf{N}(i_n), \; j_1 > M \\ j_2 \in \mathbf{N}(i_n,j_1)}} x_i^*(z_{j_2}) > C^{-1} \beta_M/2.
    \end{align*}
    Hence, as $M$ was chosen arbitrarily, we get that $(x_i^*)_{i=1}^\infty$ is unbounded, which is a contradiction.
\end{proof}

\begin{lemma} \label{LemmaOrderAnShora}
    The order of $A_n$ is at most $n+1$. Specifically $\overline{A^{(n)}_n} = A_n^{(n+1)} = \overline{A_n}^{w^*}$.
\end{lemma}

\begin{proof}
    As $\overline{A_n^{(n)}} \subseteq A_n^{(n+1)} \subseteq \overline{A_n}^{w^*}$, we just need to show that each element of $\overline{A_n}^{w^*}$ is a norm limit of elements of $A_n^{(n)}$. For $y^* \in K$ (recall that $K$ is the positive cone of $Z^*$) we define
    \begin{align*}
        \delta(y^*) &= 1 - \sum_{\substack{j_1 \in \mathbf{N}(i_n) \\ j_2 \in \mathbf{N}(i_n,j_1)}} \frac{y^*(z_{j_2})}{\beta_{j_1}} \\
        \gamma(y^*) &=  y^*(z_{i_n}) -  \sum_{\substack{j_1 \in \mathbf{N}(i_n) \\ j_2 \in \mathbf{N}(i_n,j_1)}} \frac{y^*(z_{j_2})\alpha_{j_1}}{\beta_{j_1}}.
    \end{align*}
    Take any $x^* \in K$ with finite support in $\mathbf{N}_n$ and which satisfies $\delta(x^*) > \gamma(x^*) \geq 0$. Let us consider
    \begin{align*}
        D = \operatorname{conv} \left\{ a \left( \frac{\alpha_{j_1}}{\beta_{j_1}}, \frac{1}{\beta_{j_1}} \right); \; a \geq 0, j_1 \in \mathbf{N}(i_n) \right\}.
    \end{align*}
    Then $D$ is the cone formed by rays with gradients in $[\min \alpha_{j_1},\sup \alpha_{j_1}) = [0,1)$. Hence, $(\gamma(x^*),\delta(x^*))$ is in $D$ and we can write it as a convex combination
    \begin{align*}
        (\gamma(x^*),\delta(x^*)) = \sum_{j_1 \in \mathbf{N}(i_n)} a_{j_1} \left( \frac{\alpha_{j_1}}{\beta_{j_1}}, \frac{1}{\beta_{j_1}} \right).
    \end{align*}
    We now introduce some new notation. For $j_k \in \mathbf{N}(i_n, j_1, \dots, j_{k-1})$ we define $j_k(l)$ to be the $l^{\text{th}}$ element of $\mathbf{N}(i_n,j_1,\dots,j_k)$. We write $j_k(l_1,l_2)$ instead of $(j_k(l_1))(l_2)$ for shortness. Now let us inductively define for $l,l_1,\dots,l_n \in \mathbb{N}$
    \begin{align*}
        y^*(l) &= x^* + \sum_{\substack{j_1 \in \mathbf{N}(i_{n})}} a_{j_1} z^*_{j_1(l)}\\
        y^*(l_1,l_2) &= y^*(l_1) + \sum_{\substack{j_1 \in \mathbf{N}(i_n) \\ j_2 \in \mathbf{N}(i_n,j_1)}} \left( y^*(l_1)(z_{j_2}) - \sum_{j_3 \in \mathbf{N}(i_n,j_1,j_2)} \frac{y^*(l_1)(z_{j_3}) \beta_{j_1}}{\beta_{j_2}}\right) \frac{z^*_{j_2(l_2)}\beta_{j_2}}{\beta_{j_1}}\\
        &\vdots
    \end{align*}
    \begin{align*}
        y^*(l_1,\dots,l_n) = y^*(l_1,\dots,&l_{n-1}) + \sum_{\substack{j_1 \in \mathbf{N}(i_n), \dots \\ j_n \in \mathbf{N}(i_n,j_1,\dots,j_{n-1})}} \left( y^*(l_1,\dots,l_{n-1})(z_{j_n}) - \vphantom{\sum_{j \in \mathbb{N}} \frac{\beta}{\beta}} \right. \\
        &\left. \sum_{j_{n+1} \in \mathbf{N}(i_n,j_1,\dots,j_{n})} \frac{y^*(l_1,\dots,l_{n-1})(z_{j_{n+1}}) \beta_{j_{n-1}}}{\beta_{j_n}}\right) \frac{z^*_{j_n(l_n)}\beta_{j_n}}{\beta_{j_{n-1}}}.
    \end{align*}
    It is easily proved by induction over $k = 1,\dots, n$ that $y^*(l_1,\dots,l_k)$ have finite support in $\mathbf{N}_n$.
    Further,
    \begin{align*}
        y^*(l_1,\dots,l_k) &\underset{l_k}{\overset{w^*}{\longrightarrow}} y^*(l_1,\dots,l_{k-1}) & 2 \leq k \leq n\\
        y^*(l) &\underset{l}{\overset{w^*}{\longrightarrow}} x^*.&
    \end{align*}
    To see this, consider
    \begin{align*}
        y^*(l_1,\dots,l_k) - y^*(l_1,\dots,l_{k-1}) = \sum_{\substack{j_1 \in \mathbf{N}(i_n), \dots \\ j_k \in \mathbf{N}(i_n,j_1,\dots,j_{k-1})}} c_{j_k} z^*_{j_k(l_k)},
    \end{align*}
    where
    \begin{align*}
        c_{j_k} = \left( y^*(l_1,\dots,l_{k-1})(z_{j_k}) - \sum_{j_{k+1} \in \mathbf{N}(i_n,j_1,\dots,j_k)} \frac{y^*(l_1,\dots,l_{k-1})(z_{j_{k+1}}) \beta_{j_{k-1}}}{\beta_{j_{k}}} \right) \frac{\beta_{j_k}}{\beta_{j_{k-1}}}.
    \end{align*}
    Then only finitely many $c_{j_k}$ are nonzero, as $y^*(l_1,\dots,l_{k-1})$ has finite support, and $c_{j_k}$ are independent of $l_k$. Hence, $y^*(l_1,\dots,l_k) - y^*(l_1,\dots,l_{k-1})$ is a finite linear combination of $z^*_{j_k}(l_k)$. Now, we just need to notice that the sequences $(z^*_{j_k(l_k)})_{l_k=1}^\infty$ are subsequences of $(z^*_r)_{r=1}^\infty$, which is weak$^*$ null as the basis $(z_r)_{r=1}^\infty$ is seminormalized. Similarly we get that $y^*(l)$ weak$^*$ converges to $x^*$ as $y^*(l) - x^*$ is a finite linear combination of weak$^*$ null sequences.
    
    Hence, if we prove that the elements $y^*(l_1,\dots,l_n) \in A_n$, we get that $x^* \in A_n^{(n)}$. We will prove this using Proposition \ref{PropCharakterizaceAn}. For the sake of brevity let us denote $y^* = y^*(l_1,\dots,l_n)$ and show that $y^* \in A_n$. As we have already shown that $y^*$ has finite support in $\mathbf{N}_n$, we just need to prove that the equations of Proposition \ref{PropCharakterizaceAn} hold for $y^*$. Take admissible $(i_n,j_1,\dots,j_k)$ for $2 \leq k \leq n$. Then
    \begin{align*}
        y^*(z_{j_k}) = y^*(l_1,\dots,l_{n-1})(z_{j_k}) = \cdots = y^*(l_1,\dots,l_{k-1})(z_{j_k})
    \end{align*}
    as for $k \leq m \leq n$ we have that $y^*(l_1,\dots,l_{m}) - y^*(l_1,\dots,l_{m-1})$ has support in the sets of type $\mathbf{N}(i_n,\tilde{j_1},\dots,\tilde{j_m})$ (that is indexed by sequences of length $m+1$) and $j_k \in \mathbf{N}(i_n,j_1,\dots,j_{k-1})$, which is not a set of this type as $k \leq m$. Likewise for $j_{k+1} \in \mathbf{N}(i_n,j_1,\dots,j_{k})$ we have
    \begin{align*}
        y^*(z_{j_{k+1}}) = y^*(l_1,\dots,l_k)(z_{j_{k+1}}).
    \end{align*}
    Then
    \begin{align*}
        \sum_{j_{k+1} \in \mathbf{N}(i_n,j_1,\dots,j_k)} \frac{y^*(z_{j_{k+1}})\beta_{j_{k-1}}}{\beta_{j_k}} = \sum_{j_{k+1} \in \mathbf{N}(i_n,j_1,\dots,j_k)} \frac{y^*(l_1,\dots,l_k)(z_{j_{k+1}})\beta_{j_{k-1}}}{\beta_{j_k}} \\
        = \sum_{j_{k+1} \in \mathbf{N}(i_n,j_1,\dots,j_k)} \frac{y^*(l_1,\dots,l_{k-1})(z_{j_{k+1}})\beta_{j_{k-1}}}{\beta_{j_k}} + \\ +  \left( y^*(l_1,\dots,l_{k-1})(z_{j_k}) - \sum_{j_{k+1} \in \mathbf{N}(i_n,j_1,\dots,j_k)} \frac{y^*(l_1,\dots,l_{k-1})(z_{j_{k+1}})\beta_{j_{k-1}}}{\beta_{j_k}} \right)\\
        = y^*(l_1,\dots,l_{k-1})(z_{j_k}) = y^*(z_{j_k}).
    \end{align*}
    The second equality holds by the definition of $y^*(l_1,\dots,l_k)$ and the fact that $j_k(l_k) = j_{k+1}$ for exactly one $j_{k+1} \in \mathbf{N}(i_n,j_1,\dots,j_{k})$.
    
    Recall that the coefficients $(a_{j_1})_{j_1 \in \mathbf{N}(i_n)}$ were chosen in such a way that
    \begin{align*}
        \delta(x^*) = \sum_{j_1 \in \mathbf{N}(i_n)} \frac{a_{j_1}}{\beta_{j_1}} \;\; \text{ and } \;\; \gamma(x^*) = \sum_{j_1 \in \mathbf{N}(i_n)} \frac{a_{j_1} \alpha_{j_1}}{\beta_{j_1}}.
    \end{align*}
    Hence, as $y^*(z_{j_2}) = y^*(l_1)(z_{j_2})$, we get by the definition of $y^*(l_1)$ and the fact that $z^*_{j_1(l)}(z_{j_2}) = 1$ for exactly one $z_{j_2} \in \mathbf{N}(i_n,j_1)$ and is zero otherwise that
    \begin{align*}
        \sum_{\substack{j_1 \in \mathbf{N}(i_n) \\ j_2 \in \mathbf{N}(i_n,j_1)}} \frac{y^*(z_{j_2})}{\beta_{j_1}} &=  \sum_{\substack{j_1 \in \mathbf{N}(i_n) \\ j_2 \in \mathbf{N}(i_n,j_1)}} \frac{x^*(z_{j_2})}{\beta_{j_1}} + \sum_{\substack{j_1 \in \mathbf{N}(i_n)}} \frac{a_{j_1}}{\beta_{j_1}} \\
        &= \sum_{\substack{j_1 \in \mathbf{N}(i_n) \\ j_2 \in \mathbf{N}(i_n,j_1)}} \frac{x^*(z_{j_2})}{\beta_{j_1}} + \delta(x^*) = 1.\\
        \sum_{\substack{j_1 \in \mathbf{N}(i_n) \\ j_2 \in \mathbf{N}(i_n,j_1)}} \frac{y^*(z_{j_2}) \alpha_{j_1}}{\beta_{j_1}} &=  \sum_{\substack{j_1 \in \mathbf{N}(i_n) \\ j_2 \in \mathbf{N}(i_n,j_1)}} \frac{x^*(z_{j_2})\alpha_{j_1}}{\beta_{j_1}} + \sum_{\substack{j_1 \in \mathbf{N}(i_n)}} \frac{a_{j_1}\alpha_{j_1}}{\beta_{j_1}} \\
        &= \sum_{\substack{j_1 \in \mathbf{N}(i_n) \\ j_2 \in \mathbf{N}(i_n,j_1)}} \frac{x^*(z_{j_2})}{\beta_{j_1}} + \gamma(x^*) = x^*(z_{i_n}) = y^*(z_{i_n}).
    \end{align*}
    The last equality holds as $z_{i_n}$ is not in the support of $y^* - x^*$. Therefore $y^* \in A_n$ and $x^* \in A_n^{(n)}$.
    
    Now take $z^* \in \overline{A_n}^{w^*}$. As $\overline{A_n}^{w^*} \subseteq K$, $z^*$ is norm limit of its partial sums by the virtue of Lemma \ref{LemmaKuzel}. Therefore we just need to show that the partial sums of $z^*$ are elements of $\overline{A_n^{(n)}}$. For any such partial sum $v^*$ we have that $\delta(v^*) \geq \gamma(v^*) \geq 0$, as this holds on $\overline{A_n}^{w^*}$ and taking partial sums increases $\delta$ more than it increases $\gamma$. Then $v_k^* = (1-k^{-1}) v^*$ has finite support in $\mathbf{N}_n$ and $\delta(v_k^*) > \gamma(v_k^*) \geq 0$. Hence, by the previous part of the proposition, $v_k^* \in A_n^{(n)}$, $v^* = \lim v_k^* \in \overline{A_n^{(n)}}$ and $z^* \in \overline{A_n^{(n)}}$.
\end{proof}

Now we can formulate and prove the fist main result of this paper.

\begin{theorem} \label{ThmRadn}
    Let $X$ be a non-reflexive Banach space and $n \in \mathbb{N}$. Then there is a convex subset $B \subseteq X^*$ of order $n$.
\end{theorem}

\begin{proof}
    Lemma \ref{LemmaBaze} gives us a subspace $Z$ of the space $X$ with semi-normalized basis $(z_n)_{n=1}^\infty$ with bounded partial sums. By Lemmata \ref{LemmaOrderAnZdola} and \ref{LemmaOrderAnShora} there is a convex subset $A_{n-1}$ of $Z^*$ for which $A_n^{(n-1)} \subsetneq A_n^{(n)} = \overline{A_n}^{w^*}$. Let us consider the identity embedding $E:Z \rightarrow X$ and define $B = (E^*)^{-1} (A_{n-1})$. Then Lemma \ref{LemmaPrevod} gives us
    \begin{align*}
        B^{(n-1)} \subsetneq B^{(n)} = \overline{B}^{w^*}.
    \end{align*}
\end{proof}

Now we prove that the convex set $A \subseteq Z^*$ has order $\omega + 1$. First we show that the positive cone $K$ behaves nicely with respect to restrictions on subsets of $\mathbb{N}$.

For this sake we define, for $x^* = \sum_{k=1}^\infty x^*(z_k) z_k^* \in Z^*$ and $n \in \mathbb{N}$, the restriction of $x^*$ on $\mathbf{N}_n$ as the formal sum $\sum_{k \in \mathbf{N}_n} x^*(z_k) z_k^*$. Note that in general this sum is not necessarily convergent in $Z^*$. If it is, we denote by $x^* \restriction \mathbf{N}_n$ its limit.

\begin{lemma} \label{LemmaKuzelRestrikce}
    Let $y^*$ be an element of $K$ and $n \in \mathbb{N}$. Then $y^* \restriction \mathbf{N}_n$ is a well defined element of $K$.
    
    Further, if we have a sequence $(x_k^*)_{k=1}^\infty \subseteq K$ which weak$^*$ converges to $x^* \in K$, then for all $n \in \mathbb{N}$ we have that $x_k^* \restriction \mathbf{N}_n \overset{w^*}{\underset{k}{\longrightarrow}} x^* \restriction \mathbf{N}_n$.
\end{lemma}

\begin{proof}
    For the first part we notice that $\sum_{k \in \mathbf{N}_n} y^*(z_k) z_k^*$ is a subseries of the series $\sum_{k=1}^\infty y^*(z_k) z_k^*$, which is absolutely convergent by Lemma \ref{LemmaKuzel}. Hence, it is also absolutely convergent. The fact that its limit is an element of $K$ is clear by the definitions.
    
    For the second part we first prove that the sequence $(x_k^* \restriction \mathbf{N}_n)_{k=1}^\infty$, which is well defined by the first part of the lemma, is bounded. Recall that, as the basis $(z_n)_{n=1}^\infty$ is seminormalized, the biorthogonal basic sequence $(z_n^*)_{n=1}^\infty$ is bounded by some constant $C_3 > 0$. Then for $k \in \mathbb{N}$
    \begin{align*}
        ||x_k^* \restriction \mathbf{N}_n|| \leq \sum_{l \in \mathbf{N}_n} ||x_k^*(z_l) z_l^*|| \leq C_3 \sum_{l \in \mathbf{N}_n} x_k^*(z_l) \leq C_3 \sum_{l = 1}^\infty x_k^*(z_l) \leq C \: C_3 ||x_k^*||.
    \end{align*}
    We used that $x_k^* \in K$ and Lemma \ref{LemmaKuzel}. Boundedness of $(x_k^* \restriction \mathbf{N}_n)_{k=1}^\infty$ now follows from the boundedness of the weak$^*$ convergent sequence $(x_k^*)_{k=1}^\infty$. Notice that the sequence $(x_k^* \restriction \mathbf{N}_n)_{k=1}^\infty$ converges to $x^* \restriction \mathbf{N}_n$ in the topology of pointwise convergence (that is the topology on $Z^*$ generated by $\{z_k; \; k \in \mathbb{N}\} \subseteq Z$). Hence, as the topology of pointwise convergence is a weaker Hausdorff topology then the weak$^*$ topology, they coincide on bounded subsets of $Z^*$. Therefore, as the sequence of restrictions $(x_k^* \restriction \mathbf{N}_n)_{k=1}^\infty$ is bounded, it converges to $x^* \restriction \mathbf{N}_n$ also in the weak$^*$ topology.
\end{proof}

Now let us recall that the set $A \subseteq K$ was defined as
\begin{align*}
    A = \operatorname{conv} \bigcup_{n=1}^\infty A_n.
\end{align*}
and that the sets $A_n$ have support in the sets $\mathbf{N}_n$, which form a partition of $\mathbb{N}$.

\begin{lemma} \label{LemmaARestrikce}
    Let $x^*$ be an element of $A^{(k)}$ for some $k \in \omega$. Then for all $n \in \mathbb{N}$ there is $t_n \in [0,1]$ and $x_n^* \in A_n^{(k)}$ such that $x^* \restriction \mathbf{N}_n = t_n x_n^*$ and $\sum_{n=1}^\infty t_n \leq 1$.
\end{lemma}

\begin{proof}
    We will proceed by induction. For $k = 0$ the result follows by the definition of $A$, as any $x^* \in A$ is a convex combination $\sum_{n=1}^\infty t_n x_n^*$, where $x_n^* \in A_n$. Then $x^* \restriction \mathbf{N}_n = t_n x_n^*$ as the sets $\mathbf{N}_n$ are pairwise disjoint.
    
    Now let us suppose that the lemma holds for $k \in \omega$ and take any $x^* \in A^{(k+1)}$. Then we can find a sequence $(x^*_l)_{l=1}^\infty \subseteq A^{(k)}$ which weak$^*$ converges to $x^*$. By the induction hypothesis we have
    \begin{align*}
        x_l^* \restriction \mathbf{N}_n = t_{l,n} \: x^*_{l,n}, \hspace{2cm} x_{l,n}^* \in A_n^{(k)}, \; t_{l,n} \in [0,1], \; \sum_{n=1}^\infty t_{l,n} \leq 1.
    \end{align*}
    By Lemma \ref{LemmaKuzelRestrikce} we have for each $n \in \mathbb{N}$
    \begin{align*}
        t_{l,n} \: x_{l,n}^* = x_l^* \restriction \mathbf{N}_n \overset{w^*}{\underset{l}{\longrightarrow}} x^* \restriction \mathbf{N}_n.
    \end{align*}
    Now we can, up to passing to a subsequence, assume that $t_{l,n} \underset{l}{\longrightarrow} t_n \in [0,1]$. If $t_n = 0$, we set $x_n^*$ to be any element of $A_n^{(k+1)}$. Otherwise we set $x_n^* = \frac{x^* \restriction \mathbf{N}_n}{t_n}$, which is the weak$^*$ limit of the sequence $(x_{l,n}^*)_{l=1}^\infty$. In either case we have $x^* \restriction \mathbf{N}_n = t_n x_n^*$ where $t_n \in [0,1]$ and $x_n^* \in A_n^{(k+1)}$. It remains to show that $\sum_{n=1}^\infty t_n \leq 1$. For this we use the Fatou's lemma:
    \begin{align*}
        \sum_{n=1}^\infty t_n = \sum_{n=1}^\infty \lim_{l \rightarrow \infty} t_{l,n} \leq \liminf_{l \rightarrow \infty} \sum_{n=1}^\infty t_{l,m} \leq 1.
    \end{align*}
\end{proof}

\begin{lemma}\label{LemmaOrderAZdola}
    The order of $A$ is at least $\omega + 1$.
\end{lemma}

\begin{proof}
    Consider the element $z^* = \sum_{n=1}^\infty 2^{-n} z_{i_n}^*$. Then $z^* \in \overline{A^{(\omega)}}$ as it is an infinite convex combination of the elements $z_{i_n}^*$ and by Lemma \ref{LemmaOrderAnZdola} we have that $z_{i_n}^* \in A_n^{(n+1)} \subseteq A^{(\omega)}$. Hence, we need to prove that $z^*$ is not an element of $A^{(\omega)}$, that is to prove that it is not an element of any $A^{(m)}$, $m \in \mathbb{N}$. Suppose for a contradiction that $z^* \in A^{(m)}$ for some $m \in \mathbb{N}$. Then by Lemma \ref{LemmaARestrikce} we have that
    \begin{align*}
        2^{-m-1} z_{i_{m+1}}^* =  z^* \restriction \mathbf{N}_{m+1} = t \: z_{m+1}^* \hspace{1cm} \text{for some } t \in (0,1], \; z_{m+1}^* \in A_{m+1}^{(m)}.
    \end{align*}
    In other words, $z_{i_{m+1}}^*$ is a positive multiple of an element of $A_{m+1}^{(m)}$. But then by Proposition \ref{PropRovnostiAnm} we have
    \begin{align*}
        1 = z_{i_{m+1}}^*(z_{i_{m+1}}) = \sum_{\substack{j_1 \in \mathbf{N}(i_{m+1}) \\ j_2 \in \mathbf{N}(i_{m+1},j_1)}} \frac{z_{i_{m+1}}^*(z_{j_2}) \alpha_{j_1}}{\beta_{j_1}} = 0,
    \end{align*}
    as $m \leq (m+1)-1$. But this is a contradiction. Hence, $z^* \notin A^{(\omega)}$.
\end{proof}

\begin{lemma} \label{LemmaOrderAShora}
    The order of $A$ is at most $\omega + 1$. Specifically $\overline{A}^{w^*} = \overline{A^{(\omega)}}$.
\end{lemma}

\begin{proof}
    First we notice that for each $n \in \mathbb{N}$ it holds that $0$ is an element of $A_n^{(\omega)}$ as $0 = \alpha_j z_{i_n}^* \in A_n^{(n)}$, where $j$ is the first element of $\mathbf{N}(i_n)$ (see the paragraph preceding the definition of $A_n$). 
    
    Set
    \begin{align*}
        B = \left\{ \sum_{n=1}^\infty t_n x_n^*; \; x_n^* \in \overline{A_n}^{w^*}, \; t_n \in [0,1], \; \sum_{n=1}^\infty t_n \leq 1 \right\}.
    \end{align*}
    Then $B$ is a subset of $\overline{A^{(\omega)}}$. To see this, cosider $x^* = \sum_{n=1}^\infty t_n x_n^* \in B$, where $x_n \in \overline{A_n}^{w^*}$, $t_n \in [0,1]$ and $\sum_{n=1}^\infty t_n \leq 1$. Now we will show that for each $N \in \mathbb{N}$ the partial sum $y_N^* = \sum_{n=1}^N t_n x_n^*$ is an element of $A^{(\omega)}$. By Lemma \ref{LemmaOrderAnShora} we have for each $n=1,\dots,N$ that $\overline{A_n}^{w^*} = A_n^{(N+1)}$. Hence, for these $n=1,\dots,N$ we have $x_n^* \in \overline{A_n}^{w^*} = A_n^{(N+1)} \subseteq A^{(N+1)} \subseteq A^{(\omega)}$. But then, as $0 \in A^{(\omega)}$,
    \begin{align*}
        y_N^* = \sum_{n=1}^N t_n x_n^* + \left( 1 - \sum_{n=1}^N t_n \right) \cdot 0 \in \operatorname{conv} A^{(\omega)} = A^{(\omega)}.
    \end{align*}
    Therefore $x^* = \lim_{N \rightarrow \infty} y_N^* \in \overline{A^{(\omega)}}$ and $B \subseteq \overline{A^{(\omega)}}$.
    
    Now we show that $B$ is actually already weak$^*$ closed. As $B$ is convex, it is enough to show that $B = B^{(1)}$ by the Krein-Šmulyan theorem. Let us have a sequence $(x_k^*)_{k=1}^\infty$ in $B$ which weak$^*$ converges to $x^* \in B^{(1)}$. We want to show that $x^* \in B$. As $x_k^* \in B$, we can write it as $x_k^* = \sum_{n=1}^\infty t_{k,n} x^*_{k,n}$ with $x_{k,n}^* \in \overline{A_n}^{w^*}$, $t_{k,n} \in [0,1]$ and $\sum_{n=1}^\infty t_{k,n} \leq 1$. By Lemma \ref{LemmaKuzelRestrikce} it holds for each $n \in \mathbb{N}$ that
    \begin{align*}
        t_{k,n} x_{k,n}^* = x_k^* \restriction \mathbf{N}_n \overset{w^*}{\underset{k \rightarrow \infty}{\longrightarrow}} x^* \restriction \mathbf{N}_n.
    \end{align*}
    Now we can, using the diagonal argument to pass to a subsequence if necessary, assume that for each $n \in \mathbb{N}$ it holds that $t_{k,n} \underset{k \rightarrow \infty}{\longrightarrow} t_n$ for some $t_n \in [0,1]$. Set $y_n^* = 0$ if $t_n = 0$ and otherwise set $y_n^* =\frac{x^* \restriction \mathbf{N}_n}{t_n}$, which is the weak$^*$ limit of the sequence $(x_{k,n}^*)_{k=1}^\infty$. In either case we get that $x^* \restriction \mathbf{N}_n = t_n y_n^*$, where $y_n^* \in \left(\overline{A_n}^{w^*}\right)^{(1)} = \overline{A_n}^{w^*}$, $t_n \in [0,1]$ and $\sum_{n=1}^\infty t_n \leq 1$ (where the last inequality follows again from the Fatou's lemma). Now we notice that $x^* = \sum_{n=1}^\infty \left(x^* \restriction \mathbf{N}_n \right)$, as the series $x = \sum_{n=1}^\infty x^*(z_n)z_n^*$ is absolutely convergent and the sets $\mathbf{N}_n$, $n \in \mathbb{N}$, form a partition of $\mathbb{N}$.

    Now, as obviously $A \subseteq B$, we have
    \begin{align*}
        B \subseteq \overline{A^{(\omega)}} \subseteq \overline{A}^{w^*} \subseteq \overline{B}^{w^*} = B.
    \end{align*}
    Therefore we have equalities and specifically $\overline{A}^{w^*} = \overline{A^{(\omega)}}$.
\end{proof}

Now we are all prepared to prove the second main theorem of this paper.

\begin{theorem} \label{ThmRadOmega}
    Let $X$ be a non-reflexive Banach space. Then there is a convex set $B \subseteq X^*$ of order $\omega + 1$.
\end{theorem}

\begin{proof}
    Lemma \ref{LemmaBaze} gives us a subspace $Z$ of the space $X$ with semi-normalized basis $(z_n)_{n=1}^\infty$ with bounded partial sums. By Lemmata \ref{LemmaOrderAZdola} and \ref{LemmaOrderAShora} there is a convex subset $A$ of $Z^*$ for which $A^{(\omega)} \subsetneq A^{(\omega + 1)} = \overline{A}^{w^*}$. Let us consider the identity embedding $E:Z \rightarrow X$ and define $B = (E^*)^{-1} (A_{n-1})$. Then Lemma \ref{LemmaPrevod} gives us
    \begin{align*}
        B^{(\omega)} \subsetneq B^{(\omega + 1)} = \overline{B}^{w^*}.
    \end{align*}
\end{proof}

\section{Remarks and open problems}

The order of any subset of the dual of a separable space must be a countable ordinal (see e.g. \cite{HumphreysSimpson}). It follows from the Baire category theorem, that the order of a subspace of the dual of a separable Banach space cannot be a limit ordinal. This approach, however, cannot be used for convex sets. So the following question still remains open.

\begin{question}
    Can the order of a convex subset of the dual to a separable Banach space be a limit ordinal?
\end{question}

Ostrovskii proved in \cite{ostro91} that in the dual of any non-quasi-reflexive separable Banach space we can find for any non-limit ordinal $\alpha <\omega_1$ a subspace of order $\alpha$. Can we prove an analogous statement for convex subsets of duals of non-reflexive quasi-reflexive Banach spaces?

\begin{question}
    Let $X$ be a non-reflexive quasi-reflexive Banach space. Are there any convex subsets of $X^*$ with order higher than $\omega + 1$?
\end{question}

\end{document}